\documentclass[11pt]{amsart}
\usepackage{amssymb}
\usepackage{amsfonts}
\usepackage{color}
\usepackage{xy}
\xyoption{all}

\usepackage{setspace}
\newtheorem{theorem}{Theorem}[section]
\newtheorem{proposition}[theorem]{Proposition}
\newtheorem{lemma}[theorem]{Lemma}
\newtheorem{corollary}[theorem]{Corollary}

\theoremstyle{remark}

\renewenvironment{proof}{{\noindent\bf Proof.}}{\hfill $\Box$\par\vskip3mm}

\def\NN{{\mathbb N}}

\begin{document}
\title{The Splitting Problem for Coalgebras: a Direct Approach}

\begin{abstract}
In this note we give a different and direct short proof to a previous result of Nastasescu and Torrecillas in \cite{NT} stating that if the rational part of any right $C^*$ module $M$ is a direct sumand in $M$ then $C$ must be finite dimensional (the splitting problem for coalgebras). 
\end{abstract}

\author{Miodrag Cristian Iovanov}
\thanks{$^1$}
\date{}
\maketitle

\section*{Introduction}

Let $C$ be a coalgebra over a field $k$. The category of left (resp. right) $C$-comodules is a full subcategory of the category of right (resp. left) modules over the dual algebra. In \cite{NT} it was shown that the rational part of every right $C^*$-module M is a direct summand in M if and only if C is finite dimensional. In this case, the category of rational right $C^*$-modules is equal to the category of right $C^*$-modules, and also to the Dickson subcategory of $M^C$. The aim of this note is to give a new and elementary proof of this result, based on general results on modules and comodules, and an old result of Levitzki, stating that a nil ideal in a right noetherian ring is nilpotent. The proof of Na\u as\u asescu and Torrecillas from \cite{NT} involve several techniques of general category theory (such as localization), some facts on linearly compact modules and is based on general nontrivial and profound results of Teply regarding the general splitting problem (see \cite{T1,T3}). We first prove that if C has the splitting property, that is, the rational part of every right $C^*$.-module is a direct summand, then C has only a finite number of isomorphism types of simple (left or right) comodules. We then observe that the injective envelope of every right comodule contains only finite dimensional proper subcomodules. This immediately implies that $C$ is right noetherian. Then, using a quite common old idea from Abelian group theory we use the hypothesis for a direct product of modules to obtain that every element of J, the Jacobson radical of $C$, is nilpotent. Finaly, using a well known result in noncommutative algebra due to Levitzki, we conclude that J  is nilpotent wich combined with the above mentioned key observation immediately yields that C is finite dimensional

\section{Splitting Problem}

For an $f\in C^*$, put $\overline{f}:C\rightarrow C$, $\overline{f}(x)=f(x_1)x_2$; then $\overline{f}$ is a morphism of right $C$ comodules. As a key technique, we make use of the algebra isomorphism $C^*\simeq{\rm Hom}(C^C,C^C)$ given by $f\mapsto \overline{f}$ (with inverse $\alpha\mapsto\varepsilon\circ\alpha$), where ${\rm Hom}(C^C,C^C)$ is a ring with opposite composition. Also if $T$ is a simple right $C$ subcomodule of $C$, $E(T)\subseteq C$ an injective envelope of $T$ and $C=E(T)\oplus X$ as right $C$ comodules. As $C^*\simeq E(T)^*\oplus X^*$, we identify the any element $f$ of $E(T)^*$ with the one of $C^*$ equal to $f$ on $E(T)$ and $0$ on $\Lambda$.

\begin{lemma}\label{1.fin}
If $T$ is a simple right comodule and $E(T)$ is the an injective envelope of $T$, then $E(T)$ contains only finite dimensional proper subcomodules. 
\end{lemma}
\begin{proof}
Let $K\subsetneq E(T)$ be an infinite dimensional subcomodule. Then there is a subcomodule $K\subsetneq F \subset E(T)$ such that $F/K$ is finite dimensional. We have an exact sequence of right $C^*$ modules:
$$0\rightarrow (F/K)^* \rightarrow F^* \rightarrow K^* \rightarrow 0$$
As $F/K$ is a finite dimensional rational left $C^*$ module, $(L/K)^*$ is rational right module; thus $A={\rm Rat}F^*\neq 0$. Denote $M=T^\perp\subset F^*$. Take $u\notin M$; this corresponds to some $v\in {\rm Hom}(F,C)$ such that $v\mid_{T}\neq 0$. Then $v$ is injective, because $T$ is an essential submodule of $F\subseteq E(T)$ and if ${\rm Ker}(v)\neq 0$ then ${\rm Ker}(v)\cap T\neq 0$ so ${\rm Ker}(v)\supseteq T$, which contradicts $v\mid_T\neq 0$. As $C$ is an injective right $C$ comodule and $v$ is injective we have a commutative diagram:
$$\xymatrix{
C^* \ar[r]\ar@{=}[d] & F^* \ar[r]\ar@{=}[d] & 0\\ 
{\rm Hom}^C(C,C) \ar[r]_{{\rm Hom}^C(v,C)} & {\rm Hom}^C(F,C)\ar[r] & 0 
}$$
we see that ${\rm Hom}^C(v,c)$ is generated by $v$ as ${\rm Hom}^C(C,C)\simeq C^*$ is generated by $1_C$, following that $F^*$ is generated by any $u\notin M$. Now if $F^*=A\oplus B$, we see that $A$ is finitely generated as $F^*$ is generated so it is finite dimensional, thus $A\neq F^*$ by the initial assumption. But now if $a\in A\setminus M$, $a$ generates $F^*$ so $A=F^*$, and therefore $A\subseteq M$. Also $B\neq F^*$ as $A\neq 0$ so by the same argument $B\subset M$, and therefore $F^*=A+B\subseteq M$, a contradiction ($\varepsilon\mid F\notin M$).
\end{proof}


\begin{proposition}\label{1.simples}
Let $C$ be a coalgebra such that the rational part of every finitely generated left $C^*$ module splits off. Then thare are only a finite number of isomorphism types of simple right $C$ comodules, equivalently, $C_0$ is finite dimensional.
\end{proposition}
\begin{proof}
Let $(S_i)_{i\in I}$ be the set of representatives for the simple right comodules and $\Sigma=\bigoplus\limits_{i\in I}S_i$. Then there is an injection $\Sigma\hookrightarrow C$ and we can consider $E(S_i)$ an injective envelope of $S_i$ contained in $C$. Then the sum $\sum\limits_{i\in I}E(S_i)$ is direct and there is $X<C$ such that $\bigoplus\limits_{i\in I}E(S_i)\oplus X=C$ as right $C$ comodules and left $C^*$-modules. We have $C^*=\prod\limits_{i\in I}E(S_i)^*\times X^*$, such that if $c^*\in E(S_i)^*$ and $x_j\in E(S_j)$, then $\Delta(x_j)=x_{j1}\otimes x_{j2}\in E(S_j)\otimes C$ so $c^*\cdot x_j=c^*(x_{j2})x_{j1}=0$ if $j\neq i$, as $c^*|_{E(S_j)}=0$, and the same holds if $c^*\in X^*$. Then if $c^*=((c_i^*)_{i\in I},c_X^*)$, and $c_j\in E(S_j)$ then $c^*\cdot c_j=c_j^*\cdot c_j$ ($c_j^*$ equals $c^*$ on $E(S_j)$ and $0$ otherwise). \\
Now consider $M=\prod\limits_{i\in I}S_i$ and take $x=(x_i)_{i\in I}\in M$, $x_i\neq 0$. If $y=(y_i)_{i\in I}\in M$ then for each $i$ we have $S_i=C^*\cdot x_i$ as $x_i\neq 0$ and $S_i$ is simple, so there is $c_i^*\in C^*$ such that $c_i^*\cdot x_i=y_i$. By the previous considerations, we may assume that $c_i^*\in E(S_i)^*$ (that is, it equals zero on all the components of the direct sum decomposition of $C$ except $E(S_i)$) and then there is $c^*\in C^*$ with $c^*|_{E(S_i)}=c_i^*|_{E(S_i)}$. Then one can easily see that $c^*\cdot x_i=c_i^*\cdot x_i=y_i$, thus we may extend this to $c*\cdot x=y$ showing that actualy $M=C^*\cdot x$. As $M$ is finitely generated, its rational part must split and must be finitely generated (as a direct summand in a finitely generated module), so it must be finite dimensional. But $\bigoplus\limits_{i\in I}S_i\subseteq(\prod\limits_{i\in I}S_i)$, and this shows that $I$ must be finite. As $C$ is quasifinite, this is equivalent to the fact that $C_0$ is finite dimensional.
\end{proof}

\begin{corollary}\label{1.noeth}
$C^*$ is a right noetherian ring.
\end{corollary}
\begin{proof}
Let $T$ be a right simple comodule, $E(T)\subseteq C$ an injective envelope of $T$ and $C=E(T)\oplus X$ as right $C$ comodules. If $0\neq I<E(T)*$ a right $C^*$-submodule, then for $0\neq f\in I$ put $K={\rm Ker}\overline{f}$. We have $K^\perp=\{g\in E(T)^*\mid g_{\mid K}=0\}=f\cdot C^*\subseteq I$. Indeed, if $g$ is $0$ on $K$, then $K\subseteq {\rm Ker}\overline{g}$ as $K$ is a right $C$ subcomodule of $E(T)$ and therefore it factors through $\overline{f}$: $\overline{g}=\alpha\overline{f}=\overline{h}\overline{f}=\overline{f\cdot h}$ for $h=\varepsilon\circ\alpha$, so $g=f\cdot h\in f\cdot C^*$. As $K$ is finite dimensional by Lemma \ref{1.fin}, $K^\perp=f\cdot C^*$ has finite codimension in $E(T)^*$, showing that $I\supseteq f\cdot C^*$ has finite codimension, which obviously shows that $E(T)^*$ is Noetherian. If $C_0=\bigoplus\limits_{i\in F}T_i$ with $T_i$ simple right comodules then $F$ is finite by Proposition \ref{1.simples}, so $C^*=\bigoplus\limits_{i\in F}E(T_i)^*$ is Noetherian as each $E(T_i)^*$ are.
\end{proof}

Put $R=C^*$. Note that $J=C_0^\perp=\{\overline{f}\mid \overline{f}_{\mid C_0}=0\}$ is the Jacobson radical of $R$ and $\bigcap\limits_{n\in\NN}J^n=0$. Also if $M$ is a finite dimensional right $R$-module, we have $J^nM=0$ for some $n$, because the descending chain of submodules $(MJ^n)_n$ must stationate and therefore $MJ^n=MJ^{n+1}MJ^n\cdot J$ implies $MJ^n=0$ by Nakayama lemma.

\begin{proposition}
Any element $f\in J$ is nilpotent. 
\end{proposition}
\begin{proof}
As $C$ is a finite direct sum of injective envelopes of simple right comodules $E(T)$'s, it is enough to show that $f^n_{\mid E(T)}=0$ for some $n$ for each simple right subcomodule of $C$ and injective envelope $E(T)\subseteq C$. Assume the contrary for some fixed data $T$, $E(T)$. Let 
$$M=\prod\limits_{n\geq 1}\frac{E(T)^*}{K_n^\perp}$$
where $K_n={\rm Ker}\overline{f^n}\neq E(T)$ and $K_n^\perp=\{g\in E(T)^*\mid g_{\mid K_n}=0\}$. Note that $K_n\subseteq K_{n+1}$ Put $\lambda=(f^{[n/2]}{}_{\mid E(T)})$ where $[x]$ is the greatest integer less or equal to $x$. Note that if $u$ equals $f$ on $E(T)$ and $0$ on $\lambda$ then $f^n_{\mid E(T)}$ regarded as an element of $C^*$ equals $uf^{n-1}$ (recall that we identify $E(T)^*$ as a direct summand of $C^*$). $$\lambda=(u,uf,uf,\dots,uf^{n-1},uf^{n-1},0,\dots) +(0,0,\dots,0,uf^n,uf^n,uf^{n+1},\dots)= r_n+\mu_n\cdot f^n$$ with $r_n=(u,uf,uf,\dots,uf^{n-1},uf^{n-1},0,\dots,0\dots)$ (the morphisms are always thought restrcted to $E(T)$). But then $r_n\in\prod\limits_{p\leq n}E(T)^*/K_p^\perp\times 0$ which is a rational left $C$ comodule because $E(T)^*/K_p^\perp\simeq K_p^*$ and $K_p$ is finite dimensional. Write $M=Rat M\oplus\Lambda$ as right $R$ modules and $\mu_n=q_n+\alpha_n$ with $q_n\in Rat M$ and $\alpha_n\in\Lambda$. Then if $\lambda=r+\mu$ with $r\in Rat M$ and $\mu\in\Lambda$ we have $r+\mu=r_n+\mu_n\cdot f^n=(r_n+q_n\cdot f^n)+\alpha_n\cdot f^n$ wich shows that $\mu=\mu_n\cdot f^n$. Then if $\mu=(l_p)_{p\geq 1}$ and $\mu_n=(\mu_{n,p})_{p\geq 1}$ we get that $l_p=\mu_{n,p}\cdot f^n\in E(T)^*/K_p^\perp\cdot J^p$ for all $p$ and this shows that $l_p=0$ by the previous remark so $\mu=0$. Therefore $\lambda\in Rat M$, so $\lambda\cdot R$ is finite dimensional and again we get $\lambda\cdot RJ^n=0$ for some $n$. Hence we get $f^{[p/2]+n}{}_{\mid K_p}=0,\,\forall p$, equivalently $\overline{f}^{[p/2]+n}=0$ on $K_p$ (because $K_p$ is a right comodule). For $p=2n+1$ we therefore obtain $K_{2n+1}\subseteq K_{2n}$ so $K_m=K_{m+1}$ for $m=2n$. Then if $I=Im(\overline{f}^m$, $I\neq 0$ by the assumption and there is a simle subcomodule $T'$, $T'\subseteq I$; then $\overline{f}_{\mid T'}=0$ (because $f\in J=C_0^\perp$). Take $0\neq y\in T'$; then $y=\overline{f}^m(x)$, $x\in E(T)$ and $0=\overline{f}(y)=\overline{f}^{m+1}(x)$ showing that $x\in K_{m+1}=K_m$, therefore $y=\overline{f}^m(x)=0$, a contradiction. 
\end{proof}

\begin{theorem}
If the rational part of every right $C^*$ module splits off, then $C$ is finite dimensional.
\end{theorem}
\begin{proof}
By Corollary \ref{1.noeth} $C^*$ is Noetherian and by the previous Proposition every element if $J$ is nilpotent. Therefore by Leviski's Theorem we have that $J$ is nilpotent. Now note that $C_n$ is finite dimensional for all $n$. Indeed, denoting by $s_n(M)$ the $n$-th term in the Loewy series of the comodule $M$, if $C_0=\bigoplus\limits_{i\in F}T_i$ with $T_i$ simple right comodules,  $C=\bigoplus\limits_{i\in F}E(T_i)$ with $E(T_i)$ injective envelopes of the $T_i$'s, then $C_n=\bigoplus\limits_{i\in F}s_n(E(T_i))$ and if $C_n$ is finite dimensional, then $s_{n+1}(E(T_i))$ is finite dimenional as otherwise there is a decomposition $s_{n+1}(E(T_i))/s_{n}(E(T_i))=T\oplus K$ with simple $T$ and infinite dimensional $K$ and therefore we would find an infinite dimensional subcomodule of $E(T_i)$ corresponding to $K$ which is imposible. Therefore as $J^n=0$ for some $n$ and $J^n$ has finite codimension as $J^n=C_n^\perp$ and $C_n$ is finite dimensional, we conclude that $C$ has finite dimension.
\end{proof}

\bigskip\bigskip\bigskip



\vspace*{3mm} 
\begin{flushright}
\begin{minipage}{148mm}\sc\footnotesize

Miodrag Cristian Iovanov\\
University of Southern California \\
3620 S Vermont Ave, KAP 108 \\
Los Angeles, CA 90089, USA and\\
University of Bucharest, Faculty of Mathematics, Str.
Academiei 14,
RO-70109, Bucharest, Romania\\
{\it E--mail address}: {\tt
yovanov@gmail.com, iovanov@usc.edu}\vspace*{3mm}

\end{minipage}
\end{flushright}
\end{document}